\documentclass{ws-ijbc}
\usepackage{ws-rotating}     
\usepackage{graphicx}
\usepackage{epstopdf}
\begin{document}

\catchline{}{}{}{}{} 

\markboth{Eduard Musafirov, Alexander Grin, Andrei Pranevich}{Admissible Perturbations of a Generalized Langford System}

\title{ADMISSIBLE PERTURBATIONS OF A GENERALIZED LANGFORD SYSTEM}

\author{EDUARD MUSAFIROV}
\address{Department of Technical Mechanics, Yanka Kupala State University of Grodno,\\
Ozheshko Street 22, Grodno, 230023, Belarus\\
musafirov@bk.ru}

\author{ALEXANDER GRIN}
\address{Department of Mathematical Analysis, Differential Equations and Algebra, Yanka Kupala State University of Grodno,\\
Ozheshko Street 22, Grodno, 230023, Belarus\\
grin@grsu.by}

\author{ANDREI PRANEVICH}
\address{Department of Mathematical and Information Support of Economic Systems,
Yanka Kupala State University of Grodno,\\
Ozheshko Street 22, Grodno, 230023, Belarus\\
pranevich@grsu.by}

\maketitle

\begin{history}
\received{(to be inserted by publisher)}
\end{history}

\begin{abstract}
Admissible perturbations (i.e., perturbations that do not change the Mironenko reflecting function of the system) are obtained for an autonomous three-dimensional quadratic generalized Langford system with five parameters. The obtained non-autonomous perturbed systems retain many of the qualitative properties of solutions of the original system. In particular, the instability (in the sense of Lyapunov) of the equilibrium point, the presence of a periodic solution and its asymptotic stability (instability) are proved for perturbed systems. The presence of similar chaotic attractors in the original and perturbed systems is shown by numerical simulation.
\end{abstract}


\section{Introduction}
\noindent \citet{evm:Mironenko1984} introduced the notion of the reflecting function for the qualitative investigations of the ODE system
\begin{equation} \label{evm:EQ1}
\dot{x}=X(t,x),\quad t\in {\mathbb R}, x\in D\subset {\mathbb R}^{n}
\end{equation}
under the condition that $X(t,x)$ is continuously differentiable function. This function is known now as Mironenko reflecting function (MRF) and has been efficiently applied by many authors to solve such problems of the qualitative theory of ODEs as the existence and stability of periodic solutions \cite{evm:Mironenko1989, evm:Belskii2013, evm:Liu2014, evm:Maiorovskaya2009, evm:Musafirov2008, evm:ZhouJ2020}, the existence of solutions for boundary value problems \cite{evm:Mironenko1996, evm:Musafirov2002, evm:Varenikova2012}, the solution of the center-focus problem \cite{evm:Zhou2017}, study of the global behavior of families of solutions for ODE systems \cite{evm:Mironenko2004book} and others \cite{evm:Mironenko2004book, evm:Belokurskii2013}. Moreover, it was proved that solutions of different ODE systems with the same MRF have many of the same qualitative properties \cite{evm:Mironenko2004book, evm:Mironenko2009}. Therefore, the study of the qualitative properties of solutions for a whole class of systems with the same MRF can be reduced to corresponding study of the simple (well-studied) system. In such cases non-autonomous systems \eqref{evm:EQ1} can be investigated on the base of corresponding autonomous system. In other words, an autonomous system can be perturbed into a non-autonomous systems \eqref{evm:EQ1} by using special perturbations preserving MRF which are called as \textit{admissible perturbations} (for example, admissible perturbations of the Lorenz-84 climate model were obtained by \citet{evm:Musafirov2019b}).

In this paper the describered approuch is applied for the generalized Langford system \cite{evm:Yang2018}:
\begin{equation} \label{evm:EQ2}
\begin{array}{l} {\dot{x}=ax+by+xz,} \\ {\dot{y}=cx+dy+yz,} \\ {\dot{z}=ez-\left(x^{2} +y^{2} +z^{2} \right);\quad (x,y,z)\in {\mathbb R}^{3} ,} \end{array}
\end{equation}
where $a,b,c,d,e\in {\mathbb R}$ are parameters of the system.

\citet{evm:Yang2018} analyzed the stability of equilibrium points, obtained an exact expression for a periodic orbit and some approximate expressions for limit cycles, investigated the nature of their stability, proved the existence of two heteroclinic cycles and their coexistence with a periodic orbit.

\citet{evm:Nikolov2021} considered a particular case of system \eqref{evm:EQ2} for $c=-b$, $d=a\ne 0$ and showed that system \eqref{evm:EQ2} in this case is equivalent to the nonlinear force-free Duffing oscillator $\ddot{x}+k\dot{x}+\omega x+x^{3} =0$, where $k=-(2a+e)$, $\omega =a(a+e)$. Such an equation is obtained, for example, when a steel console oscillates in an inhomogeneous field of two permanent magnets \cite{evm:Moon1979}; oscillation of a mathematical pendulum at small angles of deflection; vibrations of mass on a spring with a nonlinear restoring force located on a flat horizontal surface; and also when describing the motion of a particle in a potential of two wells and other oscillations \cite{evm:kovacic2011}. In addition, \citet{evm:Nikolov2021} proved that in this particular case, under one of three additional conditions ($e=a$ or $e=-a/2$ or $e=-2a$), the solutions of system \eqref{evm:EQ2} are expressed in explicit analytical form by means of elementary and Jacobi elliptic functions.

For the particular case of system \eqref{evm:EQ2} when $a=d=-1/3$, $b=-1$, $c=1$, $e=2/3$, the presence of chaos in the system is proved, and the chaotic attractor is also shown by \citet{evm:Belozyorov2015}.

The admissible perturbations of the non-generalized Langford system for $a=d=-2e-1$, $b=-1$, $c=1$ and for $a=d=e-1$, $b=-1$, $c=1$ were obtained by \citet{evm:Musafirov2016, evm:Musafirov2017}.

The main our purpose here is to derive a non-autonomous generalization for system \eqref{evm:EQ2} and detect qualitative properties for equilibrium points and periodic solutions of the derived system.

The structure of our paper is as follows. In section 2 we recall the definition of the MRF and basic facts for the construction of an admissible perturbations of system \eqref{evm:EQ1}. In section 3 we represent the sets of admissible perturbations of system \eqref{evm:EQ2}. In section 4 we prove the instability (in the sense of Lyapunov) of the equilibrium point $O(0,0,0)$ of admissibly perturbed systems. Section 5 presents the conditions under which admissibly perturbed systems have periodic solutions, as well as conditions for the asymptotic stability (instability) of periodic solutions. In the last section, using numerical simulations, we show similar chaotic attractors of the generalized Langford system \eqref{evm:EQ2} and an admissibly perturbed system.

\section{Brief theory of the MRF}
\noindent
First of all, we give a brief information on the theory of the MRF from \cite{evm:Mironenko2004book}.

For system \eqref{evm:EQ1}, MRF is defined as $F(t,x):=\varphi (-t;t,x)$, where $x=\varphi (t;t_{0} ,x_{0} )$ is the general solution in the Cauchy form of system \eqref{evm:EQ1}. Although the MRF is determined through the general solution of system \eqref{evm:EQ1}, it is sometimes possible to find a MRF even for non-integrable systems.

A function $F(t,x)$ is a MRF of system \eqref{evm:EQ1} if and only if it is a solution of the PDE system $\frac{\partial F}{\partial t} +\frac{\partial F}{\partial x} X(t,x)+X(-t,F)=0$ with the initial condition $F(0,x)=x$.

If the function $F(t,x)$ is continuously differentiable and satisfies the condition $F\left(-t,F(t,x)\right)\equiv F(0,x)\equiv x$, then it is the MRF of a set of systems. Moreover, all systems from this set have the same shift operator on any interval $(-\alpha ;{\kern 1pt} \alpha )$ \cite{evm:Krasnoselskii2007}. If system \eqref{evm:EQ1} is $2\omega $-periodic with respect to $t$, and $F(t,x)$ is its MRF, then $F(-\omega ,x)=\varphi (\omega ;-\omega ,x)$ is the mapping of the system over the period $[-\omega ,{\kern 1pt} \omega ]$ (Poincar\'{e} map). And therefore, all $2\omega $-periodic (with respect to $t$) systems from the set with the same MRF have the same mapping over the period $[-\omega ,{\kern 1pt} \omega ]$.

Let $2\omega $-periodic (with respect to $t$) system \eqref{evm:EQ1} and the system
\begin{equation} \label{evm:EQ3}
\dot{x}=Y(t,x),\quad t\in {\mathbb R},\; x\in D\subset {\mathbb R}^{n}
\end{equation}
have the same MRF $F(t,x)$. If the solution $\varphi (t;-\omega ,x)$ of system \eqref{evm:EQ1} and the solution $\psi (t;-\omega ,x)$ of system \eqref{evm:EQ3} are extendable to $[-\omega ,\omega ]$, then the mapping over the period $[-\omega ,\omega ]$ for system \eqref{evm:EQ1} is $\varphi (\omega ;-\omega ,x)\equiv F(-\omega ,x)\equiv \psi (\omega ;-\omega ,x)$, although system \eqref{evm:EQ3} may be non-periodic. That is, it is possible to establish a one-to-one correspondence between the $2\omega $-periodic solutions of system \eqref{evm:EQ1} and the solutions of the two-point boundary value problem $y(-\omega )=y(\omega )$ for system \eqref{evm:EQ3}.

Thanks to \citet{evm:Mironenko2009}, it became possible to find out whether two different systems of ODEs have the same MRF (in this case, the MRF itself may not be known).

\begin{theorem}[\cite{evm:Mironenko2009}]
 Let the vector functions $\Delta _{i} (t,x)$ ($i=\overline{1,m}$, where $m\in {\mathbb N}$ or $m=\infty $) be solutions of the equation
\begin{equation} \label{evm:EQ4}
\frac{\partial \Delta }{\partial t} +\frac{\partial \Delta }{\partial x} X-\frac{\partial X}{\partial x} \Delta =0
\end{equation}
and $\alpha _{i} (t)$ be any scalar continuous odd functions. Then MRF of every perturbed system of the form $\dot{x}=X(t,x)+\sum _{i=1}^{m}\alpha _{i} (t)\Delta _{i} (t,x) ,\quad t\in {\mathbb R},\; x\in D\subset {\mathbb R}^{n} $ is equal to MRF of system \eqref{evm:EQ1}.
\end{theorem}

\section{Admissible perturbations}
\noindent
For system \eqref{evm:EQ2}, we were looking for admissible perturbations of the form $\Delta \cdot \alpha (t)$, where
\[\Delta =\left(
 {\sum _{i+j+k=0}^{n} q_{ijk} x^{i} y^{j} z^{k} },\quad {\sum _{i+j+k=0}^{n} r_{ijk} x^{i} y^{j} z^{k} },\quad {\sum _{i+j+k=0}^{n} s_{ijk} x^{i} y^{j} z^{k} }
\right)^{{\rm T}} ,\]
$q_{ijk} ,r_{ijk} ,s_{ijk} \in {\mathbb R}$, $i,j,k,n\in {\mathbb N}\cup \{ 0\} $; $\alpha (t)$ is an arbitrary continuous scalar odd function. To do this, we looked for the values of the parameters $a,b,c,d,e$, $q_{ijk} ,r_{ijk} ,s_{ijk} $ for which the relation \eqref{evm:EQ4} is valid, i.e. the relation $\frac{\partial \Delta }{\partial t} +\frac{\partial \Delta }{\partial (x,y,z)} X(t,x,y,z)-\frac{\partial X(t,x,y,z)}{\partial (x,y,z)} \Delta =0$ where
 $X(t,x,y,z)=\left( {ax+by+xz}, {cx+dy+yz}, {ez-x^{2} -y^{2} -z^{2} } \right)^{{\rm T}} $
is the right-hand side of the original unperturbed system \eqref{evm:EQ2}. As a result, we were able to obtain the following statement.

\begin{theorem}
Let $\alpha _{i} (t)$ ($i=\overline{1,5}$) be arbitrary scalar continuous odd functions. Then
\begin{romanlist}
\item the MRF of system \eqref{evm:EQ2} coincides with the MRF of the system
\begin{eqnarray}
 \dot{x} & = & \left(ax+by+xz\right)\left(1+\alpha _{1} (t)\right),\nonumber \\
 \dot{y} & = & \left(cx+dy+yz\right)\left(1+\alpha _{1} (t)\right),\nonumber \\
 \dot{z} & = & \left(ez-\left(x^{2} +y^{2} +z^{2} \right)\right)\left(1+\alpha _{1} (t)\right);\nonumber
\end{eqnarray}

\item  for $c=-b$, $d=a$, the MRF of system \eqref{evm:EQ2} coincides with the MRF of the system
\begin{eqnarray}
\dot{x} & = & \left(ax+by+xz\right)\left(1+\alpha _{1} \left(t\right)\right)+x\left(a+z\right)\alpha _{2} \left(t\right)+y\alpha _{3} \left(t\right),\nonumber \\
\dot{y} & = & \left(-bx+ay+yz\right)\left(1+\alpha _{1} \left(t\right)\right)+y\left(a+z\right)\alpha _{2} \left(t\right)-x\alpha _{3} \left(t\right),\label{evm:EQ5} \\
\dot{z} & = & \left(ez-x^{2} -y^{2} -z^{2} \right)\left(1+\alpha _{1} \left(t\right)+\alpha _{2} \left(t\right)\right);\nonumber
\end{eqnarray}

\item  for $c=-b$, $d=a$, $e=-2a$, the MRF of system \eqref{evm:EQ2} coincides with the MRF of the system
\begin{eqnarray}
 \dot{x} & = & \left(ax+by+xz\right)\left(1+\alpha _{1} \left(t\right)\right)+x\left(a+z\right)\alpha _{2} \left(t\right)+y\alpha _{3} \left(t\right) \nonumber\\
   &  & -y\left(x^{2} +y^{2} \right)\left(4az+x^{2} +y^{2} +2z^{2} \right)\alpha _{4} \left(t\right), \nonumber\\
 \dot{y} & = & \left(-bx+ay+yz\right)\left(1+\alpha _{1} \left(t\right)\right)+y\left(a+z\right)\alpha _{2} \left(t\right)-x\alpha _{3} \left(t\right) \label{evm:EQ6} \\
  &  & +x\left(x^{2} +y^{2} \right)\left(4az+x^{2} +y^{2} +2z^{2} \right)\alpha _{4} \left(t\right), \nonumber\\
  \dot{z} & = & -\left(2az+x^{2} +y^{2} +z^{2} \right)\left(1+\alpha _{1} \left(t\right)+\alpha _{2} \left(t\right)\right); \nonumber
\end{eqnarray}

\item  for $c=b=0$, $d=a$, $e=-2a$, the MRF of system \eqref{evm:EQ2} coincides with the MRF of the system
\begin{eqnarray}
 \dot{x} & = & \left(ax+xz\right)\left(1+\alpha _{1} \left(t\right)\right)+y\alpha _{2} \left(t\right) \nonumber\\
   & & +y\left(4az+x^{2} +y^{2} +2z^{2} \right)\left(x^{2} \alpha _{3} \left(t\right)+xy\alpha _{4} \left(t\right)+y^{2} \alpha _{5} \left(t\right)\right), \nonumber\\
 \dot{y} & = & \left(ay+yz\right)\left(1+\alpha _{1} \left(t\right)\right)-x\alpha _{2} \left(t\right) \label{evm:EQ7} \\
    & & -x\left(4az+x^{2} +y^{2} +2z^{2} \right)\left(x^{2} \alpha _{3} \left(t\right)+xy\alpha _{4} \left(t\right)+y^{2} \alpha _{5} \left(t\right)\right), \nonumber\\
 \dot{z} & = & -\left(2az+x^{2} +y^{2} +z^{2} \right)\left(1+\alpha _{1} \left(t\right)\right).\nonumber
\end{eqnarray}
\end{romanlist}
\end{theorem}

\begin{proof}
Let us prove the second assertion of the theorem. For $c=-b$, $d=a$, the right-hand side of system \eqref{evm:EQ2} is $X=\left( {ax+by+xz}, {-bx+ay+yz}, {ez-x^{2} -y^{2} -z^{2} } \right)^{{\rm T}}$ and its Jacobi matrix is
\[\frac{\partial X(t,x,y,z)}{\partial (x,y,z)} =
\left(
\begin{array}{ccc}
 {a+z} & {b} & {x} \\
 {-b} & {a+z} & {y} \\
 {-2x} & {-2y} & {e-2z}
\end{array}
\right).\]
Let us write out the vector factors for $\alpha _{i} (t)$ from the right-hand side of system \eqref{evm:EQ5}:
\(\Delta _{1} = \left( {ax+by+xz}, {-bx+ay+yz}, {ez-x^{2} -y^{2} -z^{2} } \right)^{{\rm T}}\),
\(\Delta _{2} = \left( {x\left(a+z\right)}, {y\left(a+z\right)}, {ez-x^{2} -y^{2} -z^{2}}\right)^{{\rm T}}\),
\(\Delta _{3} = \left(  {y}, {-x}, {0} \right)^{{\rm T}}\).
By successively checking the identity \eqref{evm:EQ4} for each vector-multiplier $\Delta _{i} $ we will make sure that it is true. Let us show this, for example, for $\Delta _{2} $. The Jacobi matrix is
\[\frac{\partial \Delta _{2} }{\partial (x,y,z)} =
\left(
\begin{array}{ccc}
 {a+z} & {0} & {x} \\
 {0} & {a+z} & {y} \\
 {-2x} & {-2y} & {e-2z}
\end{array}
\right).\]
Whence we obtain
\begin{multline*}
 \frac{\partial \Delta _{2} }{\partial t} +\frac{\partial \Delta _{2} }{\partial (x,y,z)} X(t,x,y,z)-\frac{\partial X(t,x,y,z)}{\partial (x,y,z)} \Delta _{2}  \\
  \equiv \left(
  \begin{array}{c}
    {0} \\
    {0} \\
    {0}
  \end{array}
\right)+\left(
 \begin{array}{ccc}
  {a+z} & {0} & {x} \\
   {0} & {a+z} & {y} \\
   {-2x} & {-2y} & {e-2z}
  \end{array}
\right)\left(
 \begin{array}{c}
   {ax+by+xz} \\
   {-bx+ay+yz} \\
   {ez-x^{2} -y^{2} -z^{2} }
 \end{array}
\right) \\
-\left(
\begin{array}{ccc}
 {a+z} & {b} & {x} \\
 {-b} & {a+z} & {y} \\
 {-2x} & {-2y} & {e-2z}
\end{array}
\right)\left(
\begin{array}{c}
 {x\left(a+z\right)} \\
 {y\left(a+z\right)} \\
 {ez-x^{2} -y^{2} -z^{2} }
\end{array}
\right)\equiv \left(
\begin{array}{c}
 {0} \\
 {0} \\
 {0}
\end{array}
\right).
\end{multline*}
Then the second assertion of the theorem follows from Theorem 1. The rest of the statement of the theorem can be proved similarly.
\end{proof}

 When modeling real processes, the time $t\ge 0$ is usually considered, therefore the requirement that the functions $\alpha _{i} (t)$ be odd is not essential, since they can be extended continuously in an odd way to the negative time semi-axis (provided that $\alpha _{i} (0)=0$).

Theorem 2 can be used to study the qualitative behavior of the solutions of admissible perturbed systems.

\section{Instability of equilibrium point}
\noindent
By Theorem 1 \cite{evm:Yang2018}, for $e=0$, the equilibrium point $O(0,0,0)$ of system \eqref{evm:EQ2} is unstable. With this in mind, let us prove a similar statement for systems \eqref{evm:EQ5} -- \eqref{evm:EQ7}.

\begin{theorem}
Let $\alpha _{i} (t)$ ($i=\overline{1,5}$) be scalar continuous functions (not necessarily odd).
\begin{romanlist}
\item If $e=0$ and $\alpha _{1} (t)+\alpha _{2} (t)\ge l>-1$ $\forall t\ge 0$ ($l={\rm const}$), then the solution $x=y=z=0$ of system \eqref{evm:EQ5} is unstable (in the sense of Lyapunov).

\item If $a=0$ and $\alpha _{1} (t)+\alpha _{2} (t)\ge l>-1$ $\forall t\ge 0$ ($l={\rm const}$), then the solution $x=y=z=0$ of system \eqref{evm:EQ6} is unstable (in the sense of Lyapunov).

\item If $a=0$  and $\alpha _{1} (t)\ge l>-1$ $\forall t\ge 0$ ($l={\rm const}$), then the solution $x=y=z=0$ of system \eqref{evm:EQ7} is unstable (in the sense of Lyapunov).
\end{romanlist}
\end{theorem}

\begin{proof}
Consider the function $V(x,y,z)=-z^{3} $. In any neighborhood of the origin of $ {\mathbb R}^{3} $, the function $V$ is bounded and exist a region such that $V>0$.
\begin{romanlist}
\item For $e=0$, the derivative of the function $V$ along trajectories of system \eqref{evm:EQ5} is $\dot{V}=3z^{2} \left(x^{2} +y^{2} +z^{2} \right)\left(1+\alpha _{1} (t)+\alpha _{2} (t)\right)$. Since $\alpha _{1} (t)+\alpha _{2} (t)\ge l>-1$ $\forall t\ge 0$, then $\forall t\ge 0$ we have $\dot{V}\ge 3z^{2} \left(x^{2} +y^{2} +z^{2} \right)\left(1+l\right)$, where $l>-1$. Considering that $3z^{2} \left(x^{2} +y^{2} +z^{2} \right)>0$ $\forall (x,y,z)\ne (0,0,0)$, then $\dot{V}$ is positive definite function. Then, by Theorem 4.7.1 \cite{evm:Liao2007} (taking into account Corollary 4.7.3 \cite{evm:Liao2007} and its proof), the solution $x=y=z=0$ of system \eqref{evm:EQ5} is unstable.

\item For $a=0$, the derivative of the function $V$ along trajectories of system \eqref{evm:EQ6} is $\dot{V}=3z^{2} \left(x^{2} +y^{2} +z^{2} \right)\left(1+\alpha _{1} (t)+\alpha _{2} (t)\right)$. Repeating the reasoning from item (i), we find that the solution $x=y=z=0$ of system \eqref{evm:EQ6} is unstable.

\item For $a=0$, the derivative of the function $V$ along trajectories of system \eqref{evm:EQ7} is $\dot{V}=3z^{2} \left(x^{2} +y^{2} +z^{2} \right)\left(1+\alpha _{1} (t)\right)$. Since $\alpha _{1} (t)\ge l>-1$ $\forall t\ge 0$, then $\forall t\ge 0$ we have $\dot{V}\ge 3z^{2} \left(x^{2} +y^{2} +z^{2} \right)\left(1+l\right)$, where $l>-1$. Further, repeating the reasoning from item (i), we find that the solution $x=y=z=0$ of system \eqref{evm:EQ7} is unstable.
\end{romanlist}
\end{proof}

\section{Periodic solution}
\noindent
By Theorem 9 \cite{evm:Yang2018}, for $d=a$, \textit{$c=-b\ne 0$ }and $a(a+e)<0$, system \eqref{evm:EQ2} has a $2\pi /\left|b\right|$-periodic solution
\begin{eqnarray}
 x(t) & = & \sqrt{-a(a+e)} \sin \left(bt\right), \nonumber\\
 y(t) & = & \sqrt{-a(a+e)} \cos \left(bt\right),\label{evm:EQ8} \\
 z(t) & = & -a \nonumber
\end{eqnarray}
corresponding to the cycle $x^{2} +y^{2} =-a(a+e)$, $z=-a$. Moreover, this solution is asymptotically stable for $2a+e<0$ and unstable for $2a+e>0$. Similar statements are valid for systems \eqref{evm:EQ5} and \eqref{evm:EQ6}.

\begin{lemma}
Let $\alpha _{i} (t)$ ($i=\overline{1,4}$) be scalar continuous functions (not necessarily odd).
\begin{romanlist}
\item  If $a(a+e)<0$, then system \eqref{evm:EQ5} has a solution
\begin{eqnarray}
 x(t) & = &\sqrt{-a\left(a+e\right)} \sin \left(bt+\int\limits_{0}^{t} \left(b\alpha _{1} (s)+\alpha _{3} (s)\right){\rm d}s\right),\nonumber \\
  y(t) & = & \sqrt{-a\left(a+e\right)} \cos \left(bt+\int\limits_{0}^{t} \left(b\alpha _{1} (s)+\alpha _{3} (s)\right){\rm d}s\right),\label{evm:EQ9} \\
  z(t) & = & -a\nonumber
\end{eqnarray}
corresponding to the cycle  $x^{2} +y^{2} =-a(a+e)$, $z=-a$.

\item  System \eqref{evm:EQ6} has a solution
\begin{eqnarray}
  x(t) & = & a\sin \left(bt+\int\limits_{0}^{t} \left(b\alpha _{1} (s)+\alpha _{3} (s)+a^{4} \alpha _{4} (s)\right){\rm d}s\right),\nonumber \\
  y(t) & = & a\cos \left(bt+\int\limits_{0}^{t} \left(b\alpha _{1} (s)+\alpha _{3} (s)+a^{4} \alpha _{4} (s)\right){\rm d}s\right),\label{evm:EQ10} \\
  z(t) & = & -a \nonumber
\end{eqnarray}
corresponding to the cycle  $x^{2} +y^{2} =a^{2} $, $z=-a$.
\end{romanlist}
\end{lemma}

\begin{proof}
 The assertions of the lemma are proved by direct substitution of \eqref{evm:EQ9} into system \eqref{evm:EQ5} and \eqref{evm:EQ10} into system \eqref{evm:EQ6}.
\end{proof}

\begin{theorem}
 Let $\alpha _{i} (t)$ ($i=\overline{1,4}$) be scalar twice continuously differentiable odd functions, $b\ne 0$ and the right-hand sides of systems \eqref{evm:EQ5} and \eqref{evm:EQ6} be $2\pi /\left|b\right|$-periodic with respect to time $t$.
\begin{romanlist}
\item If $a(a+e)<0$ and $\exists k\in {\mathbb Z}$ such that $\int\limits_{0}^{-2\pi /\left|b\right|} \left(b\alpha _{1} (s)+\alpha _{3} (s)\right){\rm d}s=2\pi k$, then solution \eqref{evm:EQ9} of system \eqref{evm:EQ5} is $2\pi /\left|b\right|$-periodic and asymptotically stable for $2a+e<0$ and unstable for $2a+e>0$.

\item  If $\exists k\in {\mathbb Z}$ such that $\int\limits_{0}^{-2\pi /\left|b\right|} \left(b\alpha _{1} (s)+\alpha _{3} (s)+a^{4} \alpha _{4} (s)\right){\rm d}s=2\pi k$, then solution \eqref{evm:EQ10} of system \eqref{evm:EQ6} is $2\pi /\left|b\right|$-periodic.
\end{romanlist}
\end{theorem}

\begin{proof}
\begin{romanlist}
\item It follows from Theorem 2 that the MRF of system \eqref{evm:EQ5} coincides with the MRF of system \eqref{evm:EQ2} for $c=-b$ and $d=a$. By Theorem 9 \cite{evm:Yang2018}, for $d=a$, $c=-b\ne 0$ and $a(a+e)<0$, system \eqref{evm:EQ2} has a $2\pi /\left|b\right|$-periodic solution \eqref{evm:EQ8}, which is asymptotically stable for $2a+e<0$ and unstable for $2a+e>0$. By Lemma 1, system \eqref{evm:EQ5} has a solution \eqref{evm:EQ9}. Let $\Bar{\gamma}(t)=\left(x(t),y(t),z(t)\right)$ denote solution \eqref{evm:EQ8} of system \eqref{evm:EQ2} and $\bar{\chi }(t)=\left(x(t),y(t),z(t)\right)$ denote solution \eqref{evm:EQ9} of system \eqref{evm:EQ5}. If $\exists k\in {\mathbb Z}$ such that $\int\limits_{0}^{-2\pi /\left|b\right|} \left(b\alpha _{1} (s)+\alpha _{3} (s)\right){\rm d}s=2\pi k$, then $\bar{\chi }\left(-\pi /\left|b\right|\right)=\Bar{\gamma}\left(-\pi /\left|b\right|\right)$ and the statement of the theorem immediately follows from Theorem 5 \cite{evm:MironenkoV2004}.

\item It follows from Theorem 2 that the MRF of system \eqref{evm:EQ6} coincides with the MRF of system \eqref{evm:EQ2} for $c=-b$, $d=a$ and \textit{$e=-2a$}. By Theorem 9 \cite{evm:Yang2018}, for $d=a$, $c=-b\ne 0$ and $a(a+e)<0$, system \eqref{evm:EQ2} has a $2\pi /\left|b\right|$-periodic solution \eqref{evm:EQ8}. By Lemma 1, system \eqref{evm:EQ6} has a solution \eqref{evm:EQ10}. Let $\Bar{\gamma}(t)=\left(x(t),y(t),z(t)\right)$ denote solution \eqref{evm:EQ8} of system \eqref{evm:EQ2} and $\tilde{\chi }(t)=\left(x(t),y(t),z(t)\right)$ denote solution \eqref{evm:EQ10} of system \eqref{evm:EQ6}. If $\exists k\in {\mathbb Z}$ such that $\int\limits_{0}^{-2\pi /\left|b\right|} \left(b\alpha _{1} (s)+\alpha _{3} (s)+a^{4} \alpha _{4} (s)\right){\rm d}s=2\pi k$, then $\tilde{\chi }\left(-\pi /\left|b\right|\right)=\Bar{\gamma}\left(-\pi /\left|b\right|\right)$ and the statement of the theorem immediately follows from Theorem 5 \cite{evm:MironenkoV2004}.
\end{romanlist}
\end{proof}

\begin{theorem}
Let $\alpha _{i} (t)$ ($i=\overline{1,4}$) be scalar continuous functions (not necessarily odd) and $b\ne 0$.
\begin{romanlist}
\item  Let the function $b\alpha _{1} (t)+\alpha _{3} (t)$  be $2\pi /\left|b\right|$-periodic, $a(a+e)<0$, and $\int\limits_{0}^{2\pi /b} \left(b\alpha _{1} (s)+\alpha _{3} (s)\right){\rm d}s=0$, then solution \eqref{evm:EQ9} of system \eqref{evm:EQ5} is $2\pi /\left|b\right|$-periodic (the period is not necessarily minimal).

\item Let the function $b\alpha _{1} (t)+\alpha _{3} (t)+a^{4} \alpha _{4} (t)$  be $2\pi /\left|b\right|$-periodic and $\int\limits_{0}^{2\pi /b} \left(b\alpha _{1} (s)+\alpha _{3} (s)+a^{4} \alpha _{4} (s)\right){\rm d}s=0$,  then solution \eqref{evm:EQ10} of system \eqref{evm:EQ6} is $2\pi /\left|b\right|$-periodic (the period is not necessarily minimal).
\end{romanlist}
\end{theorem}

\begin{proof}
  To prove the first assertion, it suffices to prove that $\int\limits_{0}^{t+2\pi /b} \left(b\alpha _{1} (s)+\alpha _{3} (s)\right){\rm d}s\equiv \int\limits_{0}^{t} \left(b\alpha _{1} (s)+\alpha _{3} (s)\right){\rm d}s$. Taking into account that $\int\limits_{0}^{t+2\pi /b} \left(b\alpha _{1} (s)+\alpha _{3} (s)\right){\rm d}s\equiv \int\limits_{0}^{t} \left(b\alpha _{1} (s)+\alpha _{3} (s)\right){\rm d}s+\int\limits_{t}^{t+2\pi /b} \left(b\alpha _{1} (s)+\alpha _{3} (s)\right){\rm d}s$, it remains to prove that $\int\limits_{t}^{t+2\pi /b} \left(b\alpha _{1} (s)+\alpha _{3} (s)\right){\rm d}s\equiv 0$. Let us introduce the notation $A(t)=\int\limits_{t}^{t+2\pi /b} \left(b\alpha _{1} (s)+\alpha _{3} (s)\right){\rm d}s$. Since $b\alpha _{1} (t)+\alpha _{3} (t)$ is a continuous function, by the properties of an integral with a variable upper limit, $A(t)$ is a differentiable function and $\dot{A}(t)\equiv b\alpha _{1} (t+2\pi /b)+\alpha _{3} (t+2\pi /b)-\left(b\alpha _{1} (t)+\alpha _{3} (t)\right)$. Since the function $b\alpha _{1} (t)+\alpha _{3} (t)$ is \textit{$2\pi /\left|b\right|$}-periodic, it follows that $\dot{A}(t)\equiv 0$, that is, $A(t)\equiv {\rm const}$.
  In particular, $A(t)\equiv A(0)$, i.e.
\begin{equation} \label{evm:EQ11}
\int\limits_{t}^{t+2\pi /b} \left(b\alpha _{1} (s)+\alpha _{3} (s)\right){\rm d}s\equiv \int\limits_{0}^{2\pi /b} \left(b\alpha _{1} (s)+\alpha _{3} (s)\right){\rm d}s.
\end{equation}
By the hypothesis of the theorem, $\int\limits_{0}^{2\pi /b} \left(b\alpha _{1} (s)+\alpha _{3} (s)\right){\rm d}s=0$, which completes the proof of the first statement.

The second assertion of the theorem is proved similarly to the first.
\end{proof}

\begin{proposition}
In the formulation of Theorem 5:
\begin{romanlist}
\item  the condition $\int\limits_{0}^{2\pi /b} \left(b\alpha _{1} (s)+\alpha _{3} (s)\right){\rm d}s=0$ can be replaced by the condition that the function $b\alpha _{1} (t)+\alpha _{3} (t)$ is odd;

\item the condition $\int\limits_{0}^{2\pi /b} \left(b\alpha _{1} (s)+\alpha _{3} (s)+a^{4} \alpha _{4} (s)\right){\rm d}s=0$ can be replaced by the condition that the function $b\alpha _{1} (t)+\alpha _{3} (t)+a^{4} \alpha _{4} (t)$ is odd.
\end{romanlist}
\end{proposition}

\begin{proof}
 It follows from identity \eqref{evm:EQ11} for $t=-2\pi /b$ that $\int\limits_{-2\pi /b}^{0} \left(b\alpha _{1} (s)+\alpha _{3} (s)\right){\rm d}s\equiv \int\limits_{0}^{2\pi /b} \left(b\alpha _{1} (s)+\alpha _{3} (s)\right){\rm d}s$. And since the function $b\alpha _{1} (t)+\alpha _{3} (t)$ is odd, it follows that $-\int\limits_{-2\pi /b}^{0} \left(b\alpha _{1} (s)+\alpha _{3} (s)\right){\rm d}s\equiv \int\limits_{0}^{2\pi /b} \left(b\alpha _{1} (s)+\alpha _{3} (s)\right){\rm d}s$. Therefore, we have $\int\limits_{0}^{2\pi /b} \left(b\alpha _{1} (s)+\alpha _{3} (s)\right){\rm d}s=0$.

The second statement is proved similarly to the first.
\end{proof}

\begin{theorem}
  Let $\alpha _{i} (t)$ ($i=\overline{1,4}$) be scalar continuous functions (not necessarily odd) and $b=0$.
\begin{romanlist}
\item Let the function $\alpha _{3} (t)$ be $\omega $-periodic, $a(a+e)<0$ and $\exists k\in {\mathbb Z}$ such that $\int\limits_{0}^{\omega } \alpha _{3} (s){\rm d}s=2\pi k$, then solution \eqref{evm:EQ9} of system \eqref{evm:EQ5} is $\omega $-periodic (the period is not necessarily minimal).

\item Let the function $\alpha _{3} (t)+a^{4} \alpha _{4} (t)$ be $\omega $-periodic and $\exists k\in {\mathbb Z}$ such that $\int\limits_{0}^{\omega } \left(\alpha _{3} (s)+a^{4} \alpha _{4} (s)\right){\rm d}s=2\pi k$, then solution \eqref{evm:EQ10} of system \eqref{evm:EQ6} is $\omega $-periodic (the period is not necessarily minimal).
\end{romanlist}
\end{theorem}

\begin{proof}
   For \textit{$b=0$}, solution \eqref{evm:EQ9} of system \eqref{evm:EQ5} takes the form $x(t)=\sqrt{-a\left(a+e\right)} \sin \left(\int\limits_{0}^{t} \alpha _{3} (s){\rm d}s\right)$, $y(t)=\sqrt{-a\left(a+e\right)} \cos \left(\int\limits_{0}^{t} \alpha _{3} (s){\rm d}s\right)$, $z(t)=-a$, and to prove the first assertion of the theorem, it suffices to prove that $\exists k\in {\mathbb Z}$ such that $\int\limits_{0}^{t+\omega } \alpha _{3} (s){\rm d}s\equiv \int\limits_{0}^{t} \alpha _{3} (s){\rm d}s+2\pi k$. Taking into account that $\int\limits_{0}^{t+\omega } \alpha _{3} (s){\rm d}s\equiv \int\limits_{0}^{t} \alpha _{3} (s){\rm d}s+\int\limits_{t}^{t+\omega } \alpha _{3} (s){\rm d}s$, it remains to prove that $\exists k\in {\mathbb Z}$ such that $\int\limits_{t}^{t+\omega } \alpha _{3} (s){\rm d}s\equiv 2\pi k$. Let us introduce the notation $B(t)=\int\limits_{t}^{t+\omega } \alpha _{3} (s){\rm d}s$. Since $\alpha _{3} (t)$ is a continuous function, by the properties of an integral with a variable upper limit, $B(t)$ is a differentiable function and $\dot{B}(t)\equiv \alpha _{3} (t+\omega )-\alpha _{3} (t)$. Since the function $\alpha _{3} (t)$ is \textit{$\omega $}-periodic, it follows that $\dot{B}(t)\equiv 0$, that is, $B(t)\equiv {\rm const}$. In particular, $B(t)\equiv B(0)$, i.e.
\begin{equation} \label{evm:EQ12}
\int\limits_{t}^{t+\omega } \alpha _{3} (s){\rm d}s\equiv \int\limits_{0}^{\omega } \alpha _{3} (s){\rm d}s.
\end{equation}
It remains to note that, by the hypothesis of the theorem, $\exists k\in {\mathbb Z}$ such that $\int\limits_{0}^{\omega } \alpha _{3} (s){\rm d}s=2\pi k$.

The second assertion of the theorem is proved similarly to the first.
\end{proof}

\begin{proposition}
 In the formulation of Theorem 6:
\begin{romanlist}
\item the condition ``$\exists k\in {\mathbb Z}$ such that $\int\limits_{0}^{\omega } \alpha _{3} (s){\rm d}s=2\pi k$'' can be replaced by the condition that the function $\alpha _{3} (t)$ is odd;

\item the condition ``$\exists k\in {\mathbb Z}$ such that $\int\limits_{0}^{\omega } \left(\alpha _{3} (s)+a^{4} \alpha _{4} (s)\right){\rm d}s=2\pi k$'' can be replaced by the condition that the function $\alpha _{3} (t)+a^{4} \alpha _{4} (t)$ is odd.
\end{romanlist}
\end{proposition}

\begin{proof}
  From identity \eqref{evm:EQ12} for $t=-\omega $ it follows that $\int\limits_{-\omega }^{0} \alpha _{3} (s){\rm d}s\equiv \int\limits_{0}^{\omega } \alpha _{3} (s){\rm d}s$. Since the function $\alpha _{3} (t)$ is odd, it follows that $\int\limits_{-\omega }^{0} \alpha _{3} (s){\rm d}s\equiv \int\limits_{0}^{\omega } \alpha _{3} (s){\rm d}s$, then $\int\limits_{0}^{\omega } \alpha _{3} (s){\rm d}s=0$, i.e. $k=0$.

The second statement is proved similarly to the first.
\end{proof}

\section{Chaotic attractor}

By Theorem 13 \cite{evm:Yang2018}, for $c=-b$, $d=a$, $e=-2a$ and $ab\ne 0$, system \eqref{evm:EQ2} has two heteroclinic orbits connecting the equilibrium points $O(0,0,0)$ and $G(0,0,-2a)$, the eigenvalues of the Jacobi matrix for which are $\lambda _{1}^{O} =-2a$, $\lambda _{2,3}^{O} =a\pm b\sqrt{-1} $ and $\lambda _{1}^{G} =2a$, $\lambda _{2,3}^{G} =-a\pm b\sqrt{-1} $. Since $\lambda _{1}^{O} \lambda _{1}^{G} =-4a^{2} <0$ and $\Re\left(\lambda _{2}^{O} \right)\Re\left(\lambda _{2}^{G} \right)=-a^{2} <0$, the conditions of Shilnikov's Heteroclinic Theorem \cite{evm:ZhouT2006} are not satisfied. Despite this, one can expect the presence of chaos in system \eqref{evm:EQ2}, which was proved (and also showed a chaotic attractor) by \citet{evm:Belozyorov2015} for the particular case when $a=d=-1/3$, $b=-1$, $c=1$, $e=2/3$.

A numerical simulation (using Wolfram Mathematica software) shows the presence (see Fig. \ref{evm:FIG1}--\ref{evm:FIG2}) of similar chaotic attractors in systems \eqref{evm:EQ2} and \eqref{evm:EQ6} for $a=d=-3$, $b=-8$, $c=8$, $e=6$, $\alpha _{i} (t)=\sin \left(i t\right)$, $i=\overline{1,4}$. In this case, the largest Lyapunov exponent for system \eqref{evm:EQ2} is $\lambda _{\max } =0.0254794$, which confirms the chaotic nature of the attractor. To calculate the Lyapunov exponents, we used the command \(F\left[\{x\_,y\_,z\_\}\right]:=\left\{x z-3 x-8 y,8 x+y z-3 y,-x^2-y^2-z^2+6 z\right\};\ LCEsC[F, \{1/100, 2/100, 3\}, 0.05, 10000, 2, 0.01]\) from the LCE package for Wolfram Mathematica \cite{evm:Sandri1996}.

Note that if the conditions of Theorems 4 or 5 or 6 are satisfied for system \eqref{evm:EQ6}, then system \eqref{evm:EQ6} has a periodic solution, that is, system \eqref{evm:EQ6} demonstrates the coexistence of a periodic solution and a chaotic attractor.

\begin{figure}[h]
\begin{center}
\includegraphics[width=0.49\linewidth]{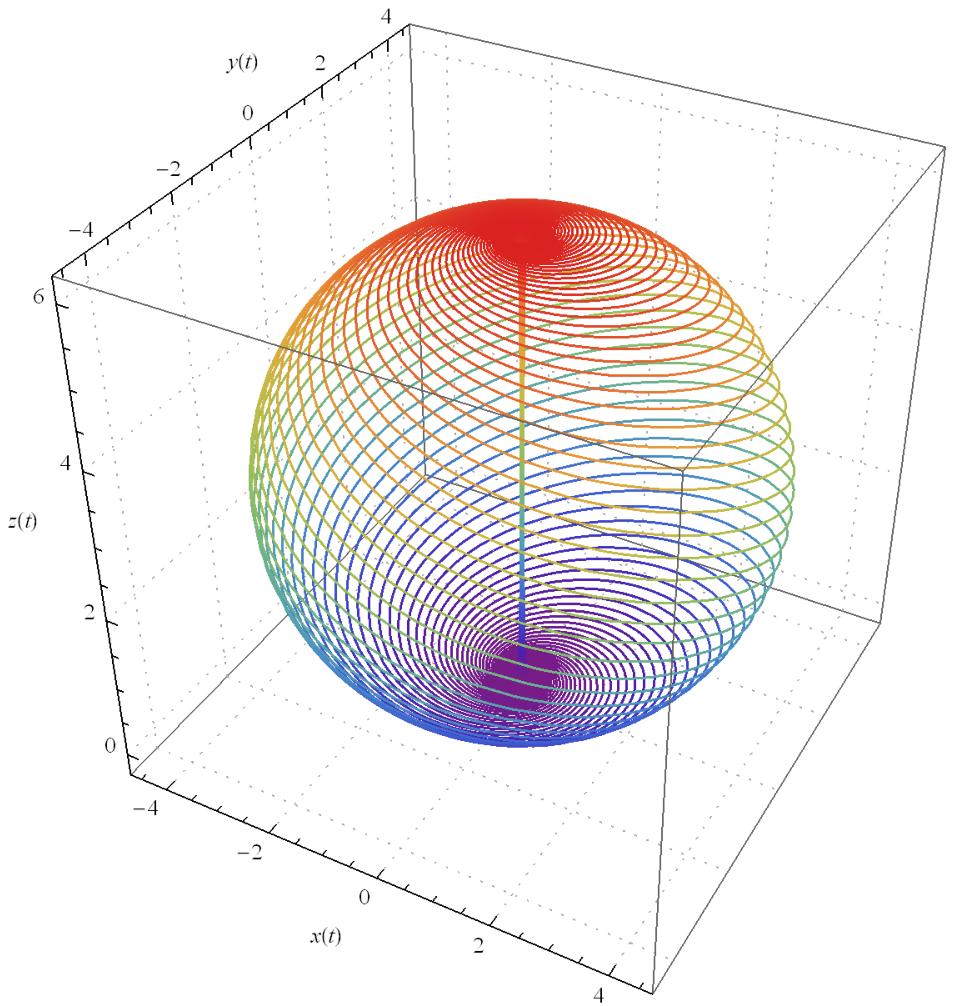}\includegraphics[width=0.49\linewidth]{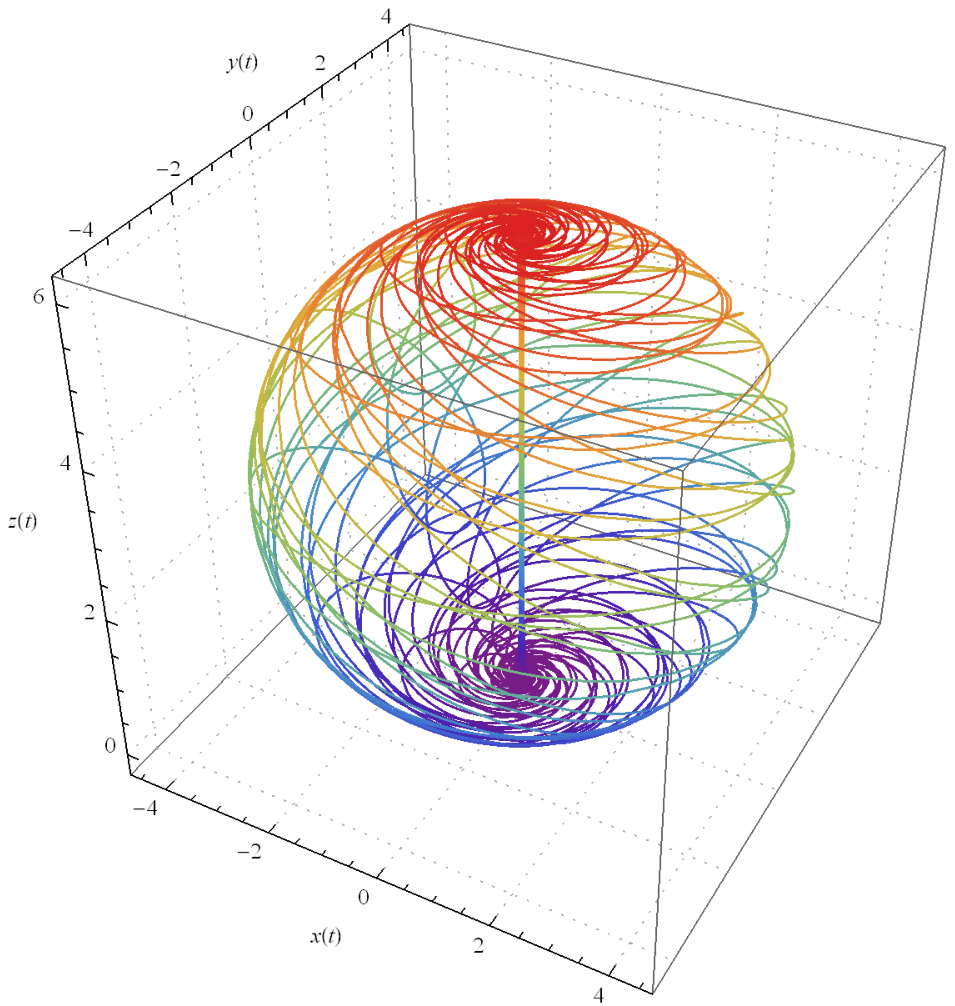}
\end{center}
\caption{Phase portraits of chaotic attractors of systems \eqref{evm:EQ2} and \eqref{evm:EQ6} (left and right, respectively) for $a=d=-3$, $b=-8$, $c=8$, $e=6$, $\alpha _{i} (t)=\sin \left(i t\right)$, $i=\overline{1,4}$.}
\label{evm:FIG1}
\end{figure}

\begin{sidewaysfigure}
\begin{center}
\includegraphics[width=0.33\linewidth]{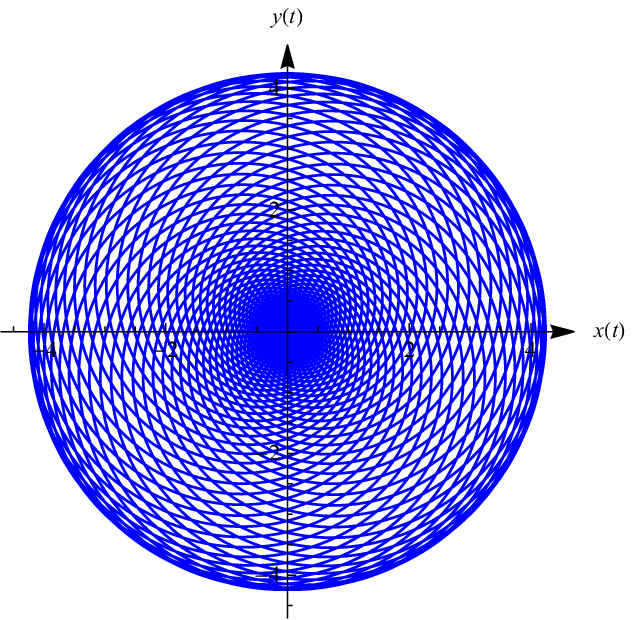}\includegraphics[width=0.33\linewidth]{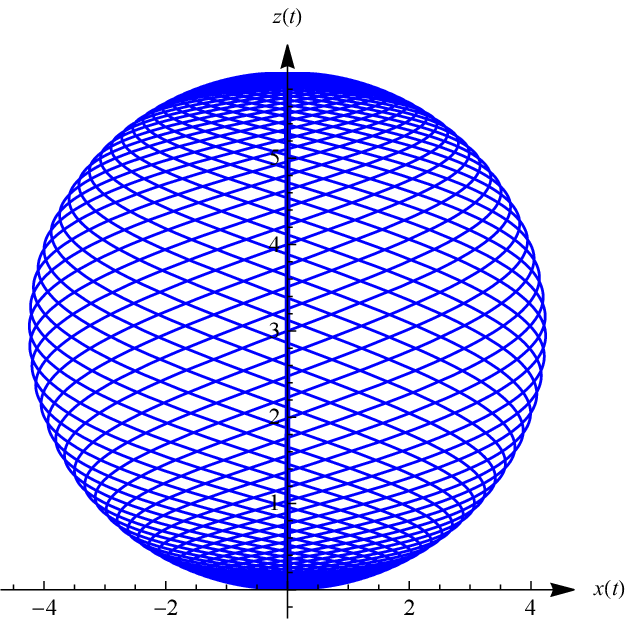}
\includegraphics[width=0.33\linewidth]{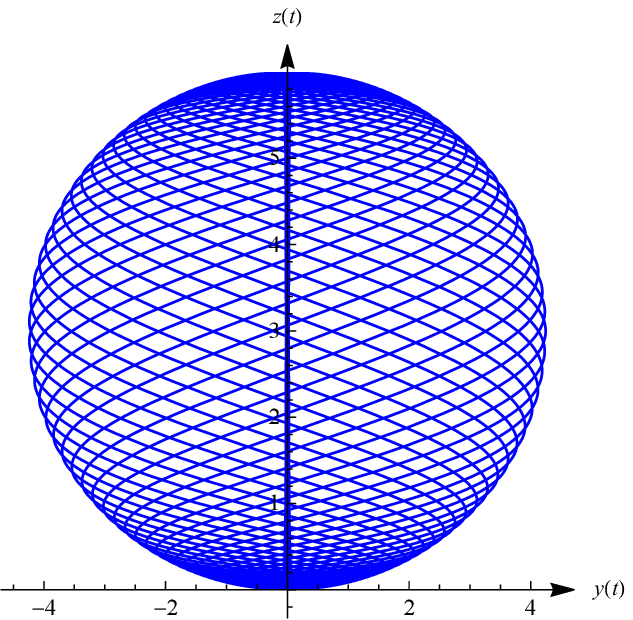}\\
\includegraphics[width=0.33\linewidth]{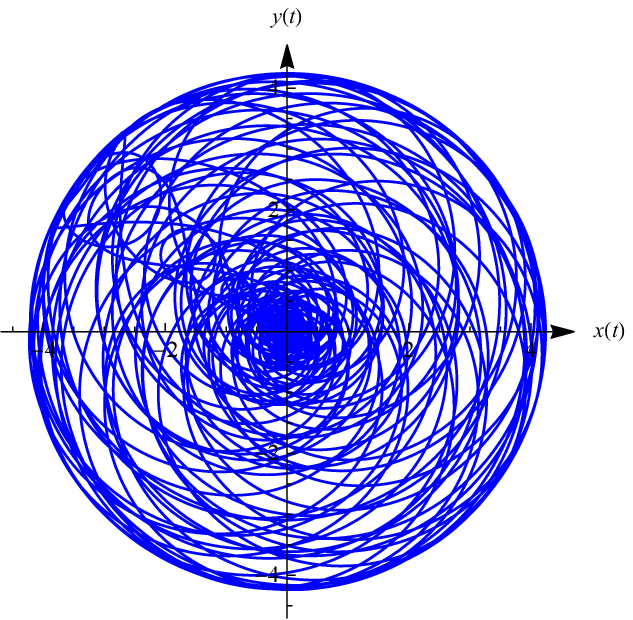}\includegraphics[width=0.33\linewidth]{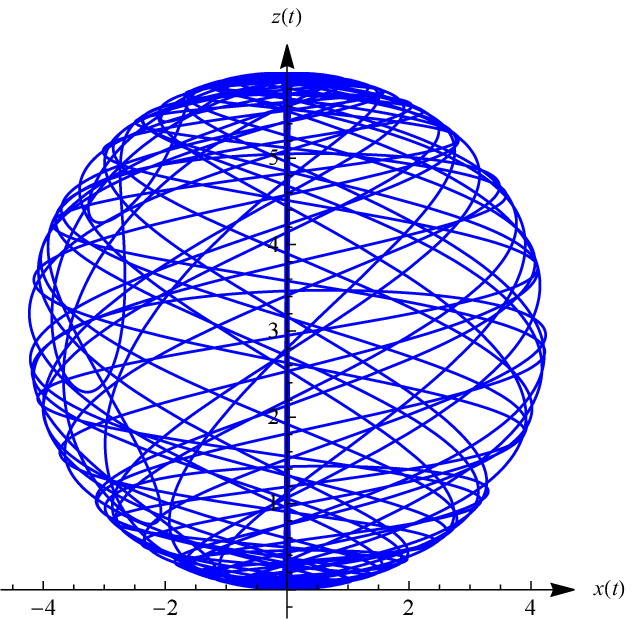}
\includegraphics[width=0.33\linewidth]{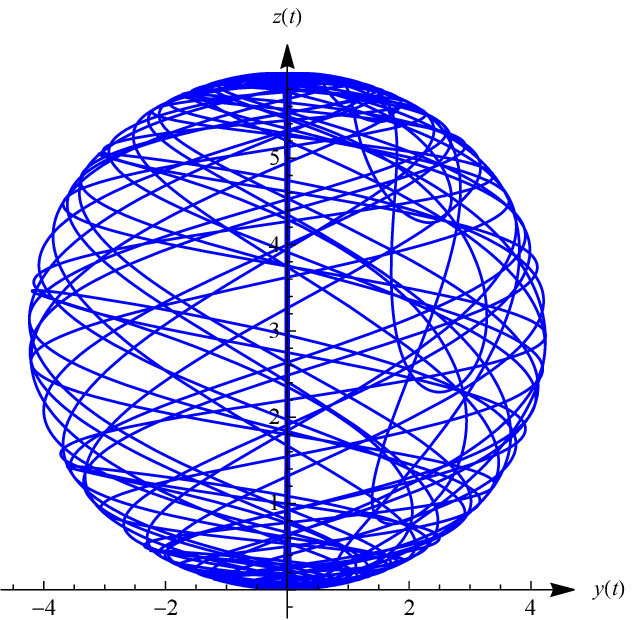}
\end{center}
\caption{Projections onto the coordinate planes of chaotic attractors of systems \eqref{evm:EQ2} and \eqref{evm:EQ6} (top and bottom row, respectively) for $a=d=-3$, $b=-8$, $c=8$, $e=6$, $\alpha _{i} (t)=\sin \left(i t\right)$, $i=\overline{1,4}$.}
\label{evm:FIG2}
\end{sidewaysfigure}

\section{Conclusion}
\noindent
A set of non-stationary systems of ordinary differential equations is obtained, the MRF of which coincides with the MRF of the autonomous generalized Langford system \eqref{evm:EQ2}. The same MRF of these systems determines the coincidence of some qualitative properties of the behavior of their solutions. This made it possible to use the results of studying the qualitative behavior of solutions of the well-studied generalized Langford system [16] to study nonstationary perturbed systems that are more complicated in kind. For such systems (\eqref{evm:EQ5}, \eqref{evm:EQ6}, and \eqref{evm:EQ7}), conditions were obtained under which the equilibrium point is unstable (in the sense of Lyapunov). For systems \eqref{evm:EQ5} and \eqref{evm:EQ6}, conditions were obtained under which these systems have periodic solutions; in addition, for system \eqref{evm:EQ5}, conditions for asymptotic stability (instability) of a periodic solution were obtained. The presence of similar chaotic attractors of systems \eqref{evm:EQ2} and \eqref{evm:EQ6} is shown using a numerical experiment. Moreover, the coexistence of a periodic solution and a chaotic attractor was shown for system \eqref{evm:EQ6}.

\nonumsection{Acknowledgments} \noindent
This research was supported by Horizon2020-2017-RISE-777911 project.

The authors are grateful to Politehnica University of Timisoara, Romania for the hospitality as well to professor Gheorghe Tigan for his support.

\bibliographystyle{ws-ijbc}
\bibliography{evm}
\end{document}